\documentclass[a4paper,11pt]{amsart}
\usepackage{amsmath,amsthm,amscd,amssymb,enumitem,latexsym,upref,hyperref,a4wide,xfrac}

\newcommand*{\mailto}[1]{\href{mailto:#1}{\nolinkurl{#1}}}

\newtheorem{theorem}{Theorem}[section]
\newtheorem{lemma}[theorem]{Lemma}
\newtheorem{corollary}[theorem]{Corollary}
\newtheorem{proposition}[theorem]{Proposition}

\theoremstyle{definition}

\newtheorem{remark}[theorem]{Remark}

\newcommand{\A}{\mathcal{A}}
\newcommand{\cE}{\mathcal{E}}
\newcommand{\E}{\mathrm{e}}
\newcommand{\I}{\mathrm{i}}

\newcommand{\re}{\mathrm{Re}}
\newcommand{\spr}[2]{\langle #1 , #2 \rangle}
\newcommand{\norm}[1]{\| #1 \|}
\newcommand{\hnorm}[1]{\| #1 \|_{\mathcal{H}}}

\newcommand{\sig}{\sigma}
\newcommand{\lam}{\lambda}
\newcommand{\om}{\omega}
\newcommand{\Om}{\Omega}

\newcommand{\R}{{\mathbb R}}
\newcommand{\N}{{\mathbb N}}

\newcommand{\h}{\mathcal{H}}
\newcommand{\D}{\mathcal{D}}
\newcommand{\DA}{\mathcal{D}(\mathcal{A})}

\newcommand{\hio}{{H_0^1(\Omega)}}
\newcommand{\hii}{{H^2(\Omega)}}

\DeclareMathOperator*{\esssup}{\mathrm{ess}\,\mathrm{sup}}

\numberwithin{equation}{section}

\begin{document}

\title[Well-posedness and asymptotic behavior]{Well-posedness and exponential decay of solutions for the 
Blackstock--Crighton--Kuznetsov equation}

\author[R.\ Brunnhuber]{Rainer Brunnhuber}
\address{Institut f\"ur Mathematik\\ Universit\"at Klagenfurt\\
Universit\"atsstrasse 65-57\\ 9020 Klagenfurt am W\"orthersee\\ Austria}
\email{\mailto{rainer.brunnhuber@aau.at}}
\urladdr{\url{http://www.aau.at/~rabrunnh/}}

\thanks{Research supported by the Austrian Science Fund (FWF): P24970}

\begin{abstract} 

The present work provides well-posedness and exponential decay results for the Blackstock--Crighton--Kuznetsov equation 
arising in the modeling of nonlinear acoustic wave propagation in thermally relaxing viscous fluids.

First, we treat the associated linear equation by means of operator semigroups. Moreover, we derive energy estimates
which we will use in a fixed-point argument in order to obtain well-posedness of the Blackstock--Crighton--Kuznetsov equation.
Using a classical barrier argument we prove exponential decay of solutions. 
\end{abstract}

\maketitle


\section{Introduction}
The present work aims to enhance the mathematical understanding of nonlinear acoustic wave propagation in viscous, thermally 
conducting, inert fluids. In particular, our motivation is to deal with higher order models arising in nonlinear acoustics.
An acoustic wave propagates through a medium as a local pressure change. Nonlinear phenomena typically 
occur at high acoustic pressures which are used for several medical and industrial purposes such as lithotripsy, 
thermotherapy, ultrasound cleaning and sonochemistry. Due to this broad range of applications, nonlinear acoustics
is currently an active field of research, see \cite{BrKa14}, \cite{BKR14}, \cite{HaBl98}, \cite{Jor04}, \cite{KaLa09}, 
\cite{KaLa11}, \cite{KaLa12}, \cite{KLM12}, \cite{KLP12}, \cite{KNT14}, \cite{Kal04}, \cite{Kuz71}, \cite{Roz08}, 
\cite{Tjo01}, \cite{Wes63} and the references therein.

The classical models in nonlinear acoustics are partial differential equations of second order in time which are 
characterized by the presence of a viscoelastic damping. The most general of these conventional models is Kuznetsov's equation
\begin{equation}
\label{kuznetsov}
u_{tt}-c^2 \Delta u - b\Delta u_t =\big(\tfrac{1}{c^2} \tfrac{B}{2A} (u_t)^2+|\nabla u|^2\big)_t,
\end{equation}
where $u$ denotes the acoustic velocity potential, $c>0$ is the speed of sound, $b\geq 0$ is the diffusivity 
of sound and $B/A$ is the parameter of nonlinearity. Neglecting local nonlinear effects 
(in the sense that the expression $c^2\vert\nabla u\vert^2-(u_t)^2$ is sufficiently small) one arrives at the Westervelt equation
\begin{equation}
\label{westervelt}
u_{tt} - b\Delta u_t -c^2\Delta u= \big(\tfrac{1}{c^2}\big(1+\tfrac{B}{2A}\big)(u_t)^2\big)_t.
\end{equation}
Both, the Kuznetsov and the Westervelt equation, can alternatively be formulated in terms of the acoustic 
pressure $p$ via the relation $\rho u_t = p$, where $\rho$ denotes the mass density. 
The quantities $A$ and $B$ occurring in the parameter of nonlinearity are the coefficients of the first and second 
order terms in the Taylor series expansion of the variation of pressure in the medium in terms of variation of the density. 
For a detailed introduction to the field of nonlinear acoustics we refer to \cite{HaBl98}.

The Kuznetsov equation in its turn can be regarded as a simplification (for a small thermal conductivity $a= \nu \mbox{Pr}^{-1}$, where
$\nu$ is the kinematic viscosity and $\mbox{Pr}$ is the Prandtl number) of the higher 
order model 
\begin{equation} 
\label{BCK}
(a\Delta - \partial_t)(u_{tt}-c^2\Delta u - b\Delta u_t) = \big(\tfrac{1}{c^2} \tfrac{B}{2A}(u_t)^2 + \left|\nabla u \right|^2 \big)_{tt}
\end{equation}
which we call Blackstock--Crighton--Kuznetsov equation. Neglecting local nonlinear effects as it is done when reducing the
Kuznetsov equation to the Westervelt equation we arrive at the Blackstock--Crighton--Westervelt equation
\begin{equation} 
\label{BCW}
(a\Delta - \partial_t)(u_{tt}-c^2\Delta u - b\Delta u_t) = \big(\tfrac{1}{c^2} \big(1+\tfrac{B}{2A}\big)(u_t)^2)_{tt}.
\end{equation}
For more information on the derivation of \eqref{BCK} and \eqref{BCW} we refer to Section \ref{sec:derivation}.

The Westervelt and the Kuznetsov equation as well as the Khoklov-Zabolotskaya-Kuznetsov equation, which is another 
standard model in nonlinear acoustics, have recently been quite extensively investigated (see, e.g., \cite{BKR14}, \cite{Jor04},
\cite{KaLa09}, \cite{KaLa12}, \cite{KaLa11}, \cite{KLM12}, \cite{KNT14} and \cite{Roz08}). 
Research on higher order models such as \eqref{BCK} and \eqref{BCW} is still in an early stage. 
The starting point was \cite{BrKa14} where well-posedness and exponential decay of solutions for \eqref{BCW} 
together with homogeneous Dirichlet boundary conditions was shown. The goal of the present paper is to provide 
results on well-posedness and exponential decay for the more general Blackstock-Crighton-Kuznetsov equation 
\eqref{BCK} which is one more step towards closing the gap of missing results on higher order models in nonlinear acoustics.

More precisely, the present work is devoted to the homogeneous Dirichlet boundary value problem 
\begin{equation}
\label{IBVP:Dir:hom}
\begin{cases}
\begin{aligned}
(a\Delta - \partial_t)(u_{tt}-b\Delta u_t - c^2 \Delta u)&= (k (u_t)^2+s|\nabla u|^2)_{tt}		&& \text{in } (0,T) \times \Omega,\\
(u,u_t,u_{tt})&=(u_0, u_1, u_2)												&& \text{on } \{t=0\}\times\Omega,\\
(u, \Delta u)&= (0,0)														&& \text{on }[0,T) \times \Gamma,
\end{aligned}
\end{cases}
\end{equation}
on a bounded domain $\Omega \subset \R^n$ of dimension $n\in \{1,2,3\}$ with smooth boundary $\Gamma=\partial\Omega$, 
where $T>0$ is either finite or infinite. The initial values $u_0, u_1, u_2\colon \Omega \rightarrow \R$ 
are given and $u\colon [0,T) \times \Omega \rightarrow \R$ is the unknown. Moreover, we assume that $a,b,c,k>0$ are constants and $s \in \{0,1\}$. 
The case $s=1$ corresponds to \eqref{BCK} whereas $s=0$ relates to \eqref{BCW}.
The restriction on the dimension of the spatial domain $\Om$ is imposed in order to be able to use the embedding 
$H^2(\Om) \hookrightarrow L_\infty(\Om)$ which we will do at several crucial steps. We therefore point out that
our results do not hold for $n\geq 4$. However, this is not of relevance in practical applications anyway.

Here, beside the classical Dirichlet boundary condition $u |_\Gamma =0$ we impose $\Delta u |_\Gamma =0$, since
this ensures that both, $a\Delta u -u_t$ and $u_{tt}-b\Delta u_t - c^2 \Delta u$, have homogeneous Dirichlet 
boundary conditions such that the homogeneous Dirichlet Laplacian can be applied. In particular, $\Delta u|_\Gamma =0$ 
allows us to interchange the differential expressions on the left-hand side which we will do when deriving energy estimates. 

Rewriting the Blackstock--Crighton--Kuznetsov equation as
\begin{equation*}
(-1-2k u_t)u_{ttt} + (a+b) \Delta u_{tt} + c^2 \Delta u_t - ab \Delta^2 u_t - ac^2 \Delta^2 u = 2k (u_{tt})^2 + 2s (|\nabla u|^2)_{tt},
\end{equation*}
we see that \eqref{IBVP:Dir:hom} degenerates whenever $u_t$ equals $-(2k)^{-1}$.
Therefore, in the analysis of \eqref{IBVP:Dir:hom} it is important to avoid degeneracy which will
here be achieved by assuming smallness of the initial data and establishing a sufficiently small 
$L_\infty(0,T;H^3(\Om))$-bound on $u_t$.
The embedding $H^3(\Om) \hookrightarrow L_\infty(\Om)$ then allows to ensure $\norm{u_t}_{L_\infty((0,T)\times \Om)} < (2k)^{-1}$ which excludes degeneracy.

The paper is organized as follows. In Section \ref{sec:derivation} we sketch the derivation of the model under consideration.
In Section \ref{sec:linear} we study the linearized version of \eqref{IBVP:Dir:hom}. We use the theory of 
operator semigroups to prove existence and uniqueness of solutions as well as their long-time behavior. 
Moreover, we derive energy estimates which are the starting point for the proof of global well-posedness for the nonlinear model.
In Section \ref{sec:wp} we show well-posedness of \eqref{IBVP:Dir:hom} for sufficiently small initial data, i.e.\ we show
existence of a unique solution of \eqref{IBVP:Dir:hom} as well as its continuous dependence on 
the (small) initial conditions (Theorem \ref{thm:wp}). The key ingredient is Banach's Fixed-Point Theorem.  
The final part of Section \ref{sec:wp} is devoted to the proof of long-time behavior of solutions (Theorem \ref{thm:decay}). Based on an energy estimate for the nonlinear equation we use a classical barrier argument in order to show exponential decay of solutions of \eqref{IBVP:Dir:hom}.

The present work extends the results from \cite{BrKa14} in several ways. On one hand, it contains well-posedness for the more general
Blackstock--Crighton--Kuznetsov equation \eqref{BCK}. On the other hand, not only for \eqref{BCK}, but also for 
the Blackstock--Crighton-Westervelt equation \eqref{BCW} we obtain strong solutions such that the respective equation is satisfied in 
$L_2(0,T;L_2(\Om))$. 
The regularity of (weak) solutions in \cite{BrKa14} is not sufficient to control the additional term $(|\nabla u|^2)_{tt}$ in the nonlinearity appearing in
\eqref{BCK}. Here, we extend the semigroup approach used in \cite{BrKa14} in order to achieve higher regularity and modify the energy estimates 
from \cite{BrKa14} in an appropriate way. Moreover, while in \cite{BrKa14} local well-posedness was shown by a combination of regularity
results for the heat equation and the Westervelt equation, we here directly use energy estimates in a fixed-point argument which provides us 
at once with (global) well-posedness for \eqref{IBVP:Dir:hom}. 

Let us mention here that, as usual, by $L_\infty(\Om)$ we denote the space of (classes of) Lebesgue-measurable 
functions $\Om \rightarrow \R$ which are essentially bounded and $L_2(\Om)$ is the Hilbert space of (classes of) 
Lebesgue square integrable functions $\Om \rightarrow \R$ equipped with the inner product $\spr{u}{v}_{L_2}=\int_\Om uv$ 
and the induced norm $\norm{u}_{L_2}$. More generally, we will always write $\|.\|_X$ for the $X$-norm of an element in a 
Banach space $X$ and $C_{X,Y}$ for the embedding constant of the continuous embedding $X\hookrightarrow Y$ of $X$
into another Banach space $Y$.

We use the Sobolev space $W_p^s(\Om)$ of order $s\in \N$ and exponent $p \in [1,\infty]$. As usual, $H^s(\Om) = W_2^s(\Om)$.
The norm of a function $u \in H^s(\Om)$, $1 \leq s \leq 4$, with zero boundary conditions in the sense
that $ u |_\Gamma =0$ for $s \in \{1,2\}$ and $u |_\Gamma = \Delta u|_\Gamma =0$ for $s \in \{3,4\}$ 
is given by $\norm{u}_{H^s}= \norm{(-\Delta)^{s/2} u}_{L_2}$, where $-\Delta$ stands for the negative 
Laplacian on $L_2(\Om)$ with domain of definition $\D(\Delta)=\{ u \in H^2(\Om)\colon u|_\Gamma =0 \}$.  
Here, $u|_\Gamma$ denotes the Dirichlet trace of $u$. 

The space $C^k([0,T];X)$ consists of all $k$-times continuously differentiable functions $u: [0,T] \rightarrow X$ 
with $k\in \N_0$. We equip it with the norm $\norm{u}_{C^k([0,T];X)}= \sum_{i=0}^k \sup_{t \in [0,T]}\norm{\partial_t^i u(t)}_X$. Furthermore,
the Sobolev space $W_p^s(0,T;X)$ consists of all functions $u \in L_p(0,T;X)$ such that $\partial_t^k u$ exists in the weak sense and
belongs to $L_p(0,T;X)$ for all $1\leq k  \leq s$. On $W_p^s(0,T;X)$ we consider the norm $\|u\|_{W_p^s(0,T;X)} = ( \int_0^T \sum_{i=0}^s \| \partial_t^i u(t) \|_X^p \, dt)^{1/p}$ if $p \in [1,\infty)$ and $\|u\|_{W_{p}^{s}(0,T;X)}= \mathrm{ess\,sup}_{t\in [0,T]} \sum_{i=0}^s \norm{\partial_t^i u}_X$ if $p=\infty$.
For additional information on function spaces and embedding theorems, we refer to \cite{AdFo03} and \cite{Eva10}.

\section{On the derivation of the Blackstock--Crighton--Kuznetsov equation}
\label{sec:derivation}
Equation \eqref{BCK} is an approximate equation which is derived from the equations of motion for viscous, thermally conducting, inert fluids of arbitrary equation of state. 
It was derived by Blackstock in \cite{Bla63}. We here point out the crucial steps of this derivation.
Four equations are needed to describe the general motion of a viscous, heat-conducting fluid: mass conservation, momentum conservation, 
entropy balance and an equation of state. The equation 
describing mass conservation is given by
\begin{equation}
\label{continuity}
\tfrac{D\rho}{Dt} + \nabla \cdot (\rho \mathbf{v}) = 0,
\end{equation}
where $\rho$ is the mass density, $\mathbf{v}$ is the fluid velocity vector and $\tfrac{D}{Dt}= \partial_t + \mathbf{v} \cdot \nabla$ is the material derivative.
Momentum conservation is governed by 
\begin{equation}
\label{momentum}
\rho \tfrac{D \mathbf{v}}{D t} + \nabla p = \mu \Delta  \mathbf{v} + \left(\mu_B+ \tfrac{1}{3} \mu\right) \nabla (\nabla\cdot  \mathbf{v}).
\end{equation}
Here, $p$ denotes the thermodynamic pressure, $\mu$ is the shear viscosity and $\mu_B$ is the bulk viscosity. Moreover,
energy conservation can be expressed as
\begin{equation}\label{entropy1}
\rho \tfrac{DE}{Dt} = a \Delta T + \mu_B (\nabla \cdot \mathbf{v})^2 + \tfrac{1}{2} \mu \big(\tfrac{\partial  v_i}{\partial x_j} \tfrac{\partial  v_j}{\partial x_i} - \tfrac{2}{3} \delta_{ij} \tfrac{\partial v_k}{\partial x_k} \big)^2,
\end{equation}
where $E$ denotes the internal energy per unit mass and $a$ is the thermal conductivity. The final term is written in Cartesian tensor notation, 
that is, ${v}_i$ denotes the component of $\mathbf{v}$ in direction $x_i$ and $\delta_{ij}$ is the Kronecker delta.
The left-hand side of \eqref{entropy1} can alternatively be expressed as 
\begin{equation}
\rho  \, \tfrac{DE}{Dt} = \rho c_v \big(\tfrac{DT}{Dt} + \tfrac{\gamma - 1}{\beta}\, \nabla \cdot \mathbf{v} \big),
\end{equation}
where $c_v$ is the specific heat at constant volume and $\beta$ is the coefficient of thermal expansion. Equation \eqref{entropy1} then becomes
\begin{equation}
\label{entropy2}
 \rho c_v \big(\tfrac{DT}{Dt} + \tfrac{\gamma - 1}{\beta}\, \nabla \cdot \mathbf{v} \big) = a \Delta T + \mu_B (\nabla \cdot \mathbf{v})^2 + \tfrac{1}{2} \mu \big(\tfrac{\partial  v_i}{\partial x_j} \tfrac{\partial  v_j}{\partial x_i} - \tfrac{2}{3} \delta_{ij} \tfrac{\partial v_k}{\partial x_k} \big)^2.
\end{equation}
Furthermore, a thermodynamic equation of state must be added, for example
\begin{equation}
\label{state}
p = p(\rho, T).
\end{equation}
Blackstock approximated the above equations of motion for viscous, thermally conducting, inert fluids so as to account as simply as possible for effects of nonlinearity and dissipation, 
see \cite{Bla63}. He followed an approximation procedure first outlined by Lighthill. In this procedure terms in the equations are classified as follows. 
Linear terms not associated with any dissipative mechanism are regarded as first-order terms. 
Second-order terms are the quadratic nonlinear terms that do not involve viscosity or heat conduction as well as the linear viscosity and
heat conduction terms. All the remaining terms are regarded as being of higher order. The basic rule in the approximation procedure is
that only first-order and second-order terms may be retained. This implies that any factor in a second-order term may be replaced by its 
first-order equivalent because any more precise substitution would result in the appearance of higher-order terms. The expressions 'basic rule'
and 'substitution corollary' were used by Blackstock in \cite{Bla63}, not originally by Lighthill. 

Blackstock carried out the approximation first for perfect gases, in which case
\eqref{state} reads $p= \rho R T$ with $R$ being the specific gas constant. Then he provided a generalization to arbitrary inert fluids.

It is convenient to work with scalar and vector potentials $u$ and $\mathbf{w}$ from the Helmholtz decomposition $\mathbf{v} = \nabla \times \mathbf{w} + \nabla u$ of 
the fluid velocity vector $\mathbf{v}$, see \cite{Bla63} and \cite{Cri79}. Only retaining the first-order terms leads to
\begin{equation*}
s_t + \Delta u = 0, \quad  
\rho_0 u_t + (p-p_0) = 0, \quad 
\mathbf{w}= \mathrm{constant}, \quad 
p-p_0 = \rho_0 c_0^2 s = \gamma^{-1} \rho_0 c_0^2 (s + \Theta),
\end{equation*}
see (4) in \cite{Bla63}. Here, we used the condensation $s= \tfrac{\rho-\rho_0}{\rho}$, a new temperature variable $\Theta = \beta_0 (T-T_0)$ and the speed of propagation $c_0$.
The subscript zero always refers to the static value of the respective expression. 

Application of the approximation procedure outlined above, equations \eqref{continuity}, \eqref{momentum}, \eqref{entropy2} and $p=\rho R T$ can be reduced
respectively to 
\begin{align}
s_t + \Delta u = c_0^{-2} (g + h),& \label{mass2} \\
\nabla ( u_t + \tfrac{p-p_0}{\rho_0} - \Lambda \nu \Delta u + \tfrac{1}{2} (\nabla u \cdot \nabla u - c_0^{-2} (u_t)^2)) + \nabla \times (\mathbf{w}_t + \nu \nabla \times \nabla \times \mathbf{w})=0,& \label{momentum2}\\
\Theta_t + (\gamma -1) \Delta u - \gamma \mathrm{Pr}^{-1} \nu \Delta \Theta = (\gamma-1) c_0^{-2} (g + (\gamma-1)h),& \label{entropy3}\\
\Theta= \gamma (p-p_0)\rho_0^{-1} c_0^{-2} - s -(\gamma-1) s^2,& \label{state2}
\end{align}
where $g = \nabla u \cdot \nabla u_t$, $h = u_t \Delta u$, $\mathrm{Pr}$ is the Prandtl number, $\nu$ is the kinematic viscosity, $c_p$ is the specific heat at constant pressure 
and $\Lambda =  \tfrac{4}{3} + \tfrac{\mu_B}{\mu}$. 
We separate \eqref{momentum2} into two equations, namely into a rotational part
\begin{equation}
\label{mom:rot}
\mathbf{w}_t + \nu \nabla \times \nabla \times \mathbf{w} =0
\end{equation}
and the time derivative of the irrotational part,
\begin{equation}
\label{mom:irro}
u_{tt} + \rho_0^{-1} p_t - \Lambda \nu \Delta u_t = -f + g.
\end{equation}
We substitute \eqref{mass2} and \eqref{mom:irro} into the time derivative \eqref{state2} which yields 
\begin{equation}
c_0^2 \Theta_t = c_0^2 \Delta u - \gamma u_{tt} + \gamma \Lambda \nu \Delta u_t - (\gamma +1) g - (\gamma-1) h.
\end{equation}
Let us differentiate \eqref{entropy3} with respect to time and plug in the latter to obtain
\begin{equation}
\label{approx:final}
\begin{aligned}
&-c_0^2 \nu \mathrm{Pr}^{-1} \Delta^2 u + \nu \mathrm{Pr}^{-1} \Delta u_{tt} + (\Lambda + (\gamma-1) \mathrm{Pr}^{-1}) \nu \Delta u_{tt} \\
&- \nu^2 \mathrm{Pr}^{-1} \gamma \Lambda \Delta^2 u_t + c_0^2 \Delta u_t - u_{ttt} 
= \big( \tfrac{\gamma-1}{2c_0^2} (u_t)^2+ |\nabla u|^2\big)_{tt},
\end{aligned}
\end{equation}
where we have applied the basic rule to eliminate nonlinear dissipation terms. In contrast to (7) in \cite{Bla63}, 
we here retained $- \nu^2 \mathrm{Pr}^{-1} \gamma \Lambda \Delta^2 u_t$. The term $b= (\Lambda + (\gamma-1) \mathrm{Pr}^{-1}) \nu$
is referred to as acoustic diffusivity or diffusivity of sound (Lighthill). In case of an arbitrary inert fluid we replace $\gamma-1$ on the right-hand side
by the parameter of nonlinearity $B/A$, see Section 4.2 in \cite{Bla63}. Upon the coefficient of the $\Delta^2 u_t$-term, \eqref{approx:final} equals \eqref{BCK}. 
If we neglect boundary dissipation in the sense that wall effects are not important, we may approximate the flow as irrotational. In this case
\eqref{mom:rot} has the trivial solution $\mathbf{w}= \mathrm{constant}$ and is therefore of no further concern, see page 23 in \cite{Bla63}. 
The flow is then completely defined by \eqref{approx:final}. 

If temperature conditions are not significant, the fourth-order character of \eqref{approx:final} need not be preserved. In this case, \eqref{approx:final} 
can be reduced to Kuznetsov's equation which is considerably simpler, see page 23 in \cite{Bla63}. However, if the temperature constraints need to be retained, one needs to consider the fourth-order equation \eqref{approx:final}.

\section{The linearized initial boundary value problem}
\label{sec:linear}
Before we turn to the nonlinear analysis, we consider the linearized version of \eqref{IBVP:Dir:hom}, i.e.
\begin{equation}
\label{IBVP:Dir:hom:lin}
\begin{cases}
\begin{aligned}
(a\Delta - \partial_t)(u_{tt}-b\Delta u_t - c^2 \Delta u)&= f		&& \text{in }(0,T) \times \Omega,\\
(u,u_t,u_{tt})&=(u_0, u_1, u_2)							&& \text{on } \{t=0\}\times\Omega,\\
(u, \Delta u)&= (0,0)									&& \text{on } [0,T) \times \Gamma,
\end{aligned}
\end{cases}
\end{equation}
where $f\colon (0,T)\times \Omega \rightarrow \R$. 
In the present section shall consider \eqref{IBVP:Dir:hom:lin} in a general abstract form. More precisely, investigate the abstract 
linear initial value problem
\begin{equation}
\label{IBVP:abstract}
\begin{cases}
\left(-a\mathcal{A} - \partial_t \right) \left(u_{tt}(t)+c^2\mathcal{A}u(t) + b\mathcal{A}u_t(t) \right)=f(t), & t \in (0,T)\\
u(0)=u_0, \quad u_t(0)=u_1, \quad u_{tt}(0)=u_2 &
\end{cases}
\end{equation}
defined on a separable Hilbert space $\h$ which is endowed with the inner product $\spr{.}{..}_\h$ and the 
induced norm $\hnorm{.}$. Here, $\A\colon \DA \rightarrow \h$ is assumed to be a self-adjoint and closed linear operator 
whose domain of definition $\DA$ as well as $\D(\A^2)$ is dense in $\h$. Moreover, we suppose that the spectrum $\sig(\A)$ of $\A$ 
is contained in $(0,\infty)$ and consists only of eigenvalues of $\A$ with finite multiplicity. 
As a consequence, there exists some $\lam_0>0$ such that $\sig(\A) \subset [\lam_0, \infty)$. We will always set
$\lam_0 = \min \sig(\A)$. 
Moreover, since $\sigma(\A)$ is strictly positive, for $\Theta >0$ fractional powers $\A^\Theta$ of $\A$ are well-defined 
and $\A^\Theta$ is again a self-adjoint linear operator with domain of definition $\D(\A^\Theta)$ such that 
$\sigma(\A^\Theta) \subset (0,\infty)$. Note that $\D(\A^{\Theta_1}) \subset \D(\A^{\Theta_2})$ for $\Theta_1 > \Theta_2$.
We assume the embeddings
\begin{align}\label{Poincare}
\D(\A^{\nu}) &\hookrightarrow \h, \qquad&& \mbox{with } \hnorm{v} \leq C_{\D(\A^\nu), \h} \hnorm{\A^\nu v}, \nu=\tfrac12, 1, \tfrac32, \\
\DA &\hookrightarrow L_\infty(\Om), && \mbox{with } \norm{v}_{L_\infty} \leq C_{\DA,L^\infty} \hnorm{\A v}. \label{linfDA}
\end{align}
to hold. Clearly, in connection with \eqref{IBVP:Dir:hom:lin}, we always think of $\h=L_2(\Omega)$ and $\A= -\Delta$ being the 
negative Dirichlet Laplacian with domain of definition $\D(\Delta)=\{u \in H^2(\Om)\colon u|_\Gamma=0\}$. In this context, \eqref{Poincare} corresponds to 
the Poincar\'e inequality, i.e. to the (repeatedly used) embedding $H_0^1(\Om) \hookrightarrow L_2(\Om)$, and \eqref{linfDA} 
corresponds to the embedding $H^2(\Om) \cap H_0^1(\Om) \hookrightarrow L_\infty(\Om)$.

\subsection{Semigroup framework}
In this section we investigate the linearized problem associated with \eqref{IBVP:Dir:hom} by means of the
theory of operator semigroups. This has already been done in \cite[Section 3]{BrKa14} and resulted in
existence of a unique solution $u \in C^1([0,T]; \hii \cap \hio) \cap C^2([0,T]; L_2(\Om))$ of \eqref{IBVP:Dir:hom:lin} 
provided $u_0, u_1 \in \hii \cap \hio$ and $u_2 \in L_2(\Om)$ and $f \in C^\alpha([0,T]; L_2(\Om))$ is H\"older continuous
with exponent $\alpha \in (0,1)$, see Corollaries 3.14 and 3.16 in \cite{BrKa14}.
Here our aim is to obtain solutions with higher regularity which can be achieved by modifying 
the action space of the semigroup and the domain of definition of its generator. 
Like in \cite{BrKa14}, we represent \eqref{IBVP:abstract} as an abstract parabolic evolution equation
\begin{equation}
\label{parabolic}
U_t(t) = A U(t) + F(t), \, t \in (0,T), \qquad U(0) = U_0,
\end{equation}
by choosing
\begin{equation}\label{initialcondmatrix}
U(t)= \begin{pmatrix} u(t) \\ u_t(t) \\ u_{tt}(t) +b\A u_t(t) + c^2 \A u(t)\end{pmatrix}, \qquad
U_0 = \begin{pmatrix} u_0 \\ u_1 \\ u_2  + b \A u_1+ c^2 \A u_0 \end{pmatrix}, 
\end{equation}
and $F(t) = (0,0,f(t))^\top$, but here consider the action space $H = \D(\A^2) \times \DA \times \h$ endowed with the inner product
\begin{equation} \label{spr}
\spr{v}{w}_H = \big(\tfrac{\alpha b}{2} \big)^2 \spr{\A^2 v_1}{\A^2 w_1}_\h + \spr{\A v_2}{\A w_2}_\h + \spr{v_3}{w_3}_\h
\end{equation}
for $v = (v_1, v_2, v_3)$ and $w=(w_1, w_2, w_3)$ in $H$ and the closed linear operator 
$A\colon \D(A) \rightarrow H$ defined by
\begin{equation}\label{generator1}
A= \begin{pmatrix} 0 & I & 0 \\ -c^2 \A & -b\A  & I \\  0 & 0 & -a \A \end{pmatrix},\qquad
\D(A)=\D(\A^2) \times \D(\A^2) \times \DA.
\end{equation}
The constant $\alpha >0$ in \eqref{spr} will be chosen appropriately during the proof of the next result.
\begin{proposition}
The operator $A\colon \D(A) \rightarrow H$ given by \eqref{generator1} generates an analytic semigroup on $H$.
\end{proposition}
\begin{proof}
The proof follows \cite[Section 3.1.2]{BrKa14}. Here we just outline the main steps. We decompose $A$, $A= A_1 + A_2$, where
\begin{equation*}
A_1 = \begin{pmatrix} 0 & 0 & 0 \\ 0 & -b\A & 0 \\ 0 & 0 & -a\A \end{pmatrix} 
\quad \text{and} \quad 
A_2 = \begin{pmatrix} 0 & I & 0 \\ -c^2\A & 0 & I \\ 0 & 0 & 0 \end{pmatrix}. 
\end{equation*}
For $v, w \in \D(A)$ we have $\spr{A_1 v}{w}_H = \spr{v}{A_1 w}_H$. Moreover, $A_1 \colon \D(A) \rightarrow H$ 
is densely defined and closed, hence there exists a self-adjoint extension of $A_1$. 
Since $\sigma(A_1) \subset (-\infty, 0]$ the analyticity of the semigroup generated by $A_1$ 
follows at once from \cite[Corollary II.4.7]{EnNa00}.
Moreover, $A_2\colon H \rightarrow H$ is relatively $A_1$-bounded with $A_1$-bound $\tfrac{\alpha}{2} < \alpha$. 
In particular, we have $\norm{A_2 v}_H = \tfrac{\alpha}{2} \norm{A_1 v}_H + \sqrt{2} \max \{\tfrac{2c^2}{\alpha b}, 1 \} \norm{v}_H$ 
for all $v \in \D(A)$. The latter is derived analogously to (3.21) in \cite{BrKa14} using $\norm{\cdot}_H = (\spr{\cdot}{\cdot}_H)^{1/2}$.
The analyticity of the semigroup generated by $A\colon \D(A) \rightarrow H$ follows from \cite[Theorem III.2.10]{EnNa00} 
provided $\alpha >0$ in \eqref{spr} is chosen sufficiently small.
\end{proof}

\begin{remark}
The presence of the damping parameter $b>0$ is essential in order to obtain analyticity of the semigroup 
generated by $A\colon \D(A) \rightarrow H$. If $b = 0$, the spectrum of $A$ is given by 
$\sig(A) = \{-a\lam_n,\pm \I c \lam_n \colon \lam_n \in \sig(\A)\}$. Hence, in this case, $A$ is not a sectorial operator 
and can thus not be the generator of an analytic semigroup, cf. Remark 3.11 in \cite{BrKa14}.
\end{remark}

\begin{proposition}
Suppose $U_0 \in \D(A)$. Let moreover $f \in C^\alpha([0,T]; \h)$
be H\"older continuous with exponent $\alpha \in (0,1)$. Then the abstract initial value problem \eqref{parabolic}
has a unique solution $U \in C([0,T];\D(A)) \cap  C^1([0,T];H)$
which is given by 
\begin{equation*}
U(t) = \E^{tA} U_0 + \int_0^t \E^{(t-s)A} F(s) \, ds, \qquad 0 \leq t \leq T.
\end{equation*}
\end{proposition}
\begin{proof}
It suffices to observe that $AU_0 +F(0)$ is in the closure of $\D(A)$ and use \cite[Theorem 4.3.1]{Lun95} on
the existence of a unique strong solution of inhomogeneous parabolic problems.
\end{proof}

\begin{corollary}
\label{cor:homsemigroup}
Suppose $u_0, u_1 \in H^4(\Om)$ and $u_2 \in \hii$ such that 
$u_0 |_\Gamma = \Delta u_0  |_\Gamma= u_1  |_\Gamma = \Delta u_1 |_\Gamma = u_2  |_\Gamma = 0$ in the sense of traces. Let moreover $f \in C^\alpha([0,T]; \h)$
be H\"older continuous with exponent $\alpha \in (0,1)$. Then the linear boundary value problem \eqref{IBVP:Dir:hom:lin}
has a unique solution
\begin{equation*}
u \in  C^1([0,T];H^4(\Om)) \cap C^2([0,T];\hii) \cap C^3([0,T]; L_2(\Om))
\end{equation*}
such that $u(t) |_\Gamma = \Delta u(t)  |_\Gamma= u_t(t)  |_\Gamma = \Delta u_t(t) |_\Gamma = u_{tt} (t) |_\Gamma = 0$ in the sense of traces for all $t \in (0,T]$.
\end{corollary}
\begin{proof}
We now use the choices $\A= - \Delta$, $D(\A) = \{ u \in H^2(\Om)\colon u|_\Gamma =0\}$ and $\h = L_2(\Om))$. 
Therewith $u_0 \in \D(\A^2)$ results in $u_0 \in H^4(\Om)$ with $u_0 |_\Gamma = \Delta u_0 |_\Gamma =0$. The 
same holds for $u_1$. Using the regularity of $u_0$ and $u_1$, we see that $u_2 + b\A u_1 + c^2 \A u_0 \in \DA$ 
gives us $u_2 \in H^2(\Om)$ such that $u_2 |_\Gamma = 0$. 
Furthermore, it is straightforward to check that $U \in C([0,T];\D(A)) \cap C^1([0,T];H)$ implies $u \in C([0,T]; \D(\A^2) \cap C^1([0,T];\D(\A^2))$, $u_t \in C([0,T]; \D(\A^2) \cap C^1([0,T];\D(\A))$ and 
$u_{tt} +b \A u_t + c^2 \A u \in C([0,T]; \D(\A) \cap C^1([0,T];\h)$. Therewith, again using the explicit choices for $\A$ and $\h$,  
we obtain the regularity as well as the traces of $u$ as stated.
\end{proof}

The spectral bound $s(A)=\{\re(\lam)\colon \lam \in \sigma(A)\}$ of $A$ is given by $s(A)=-\min \{a\lam_0, \tfrac{b}{2}\lam_0,\tfrac{c^2}{b}\}$,
where $\lam_0=\min \sigma(\A)$, see Lemma 3.12 in \cite{BrKa14}. As a consequence, we obtain an 
exponential decay result for the homogeneous equation.
\begin{proposition}
Let $u_0, u_1$ and $u_2$ be given as in Corollary \ref{cor:homsemigroup} and suppose $f\equiv 0$. Then
any solution $u$ of \eqref{IBVP:Dir:hom:lin} decays exponentially fast to zero in the sense that there exist 
$M, \om >0$ such that 
\begin{equation*}
E[u](t) \leq M \E^{-\om t} E[u](0),\qquad t>0
\end{equation*}
where $E[u](t) = \|u(t)\|_{H^4}^2 + \|u_t(t)\|_{H^4}^2 + \|u_{tt}(t)-b\Delta u_t(t) -c^2\Delta u(t)\|_{H^2}^2$.
\end{proposition}
\begin{proof}
Since $s(A) <0$, the semigroup $\E^{tA}$ generated by $A$ is uniformly exponentially stable. Hence there are constants 
$\tilde{M} > 1$ and $\omega >0$ such that $\norm{\E^{tA}U(0)}_H  = \norm{U(t)}_H \leq \tilde{M} \E^{-\om t} \norm{U(0)}_H$. Rescaling the latter 
immediately yields the claim.
\end{proof}

\subsection{Energy estimates}
\label{sec:energy}
Our second approach to treat the linear equation is to derive energy estimates of which we will also further make use of when
proving global well-posedness of our nonlinear Dirichlet boundary value problem \eqref{IBVP:Dir:hom}.
For now we consider \eqref{IBVP:Dir:hom:lin} or, in a more abstract form \eqref{IBVP:abstract}. We assume for the moment that $u$ is sufficiently
smooth and interchange the order of differentiation on the left-hand side of our PDE, that is, we consider
\begin{equation}
\label{IBVP:abstract:interchange}
\begin{cases}
(\partial_t^2+b\A \partial_t +c^2\A)(a\A u + u_t)= -f, & t \in (0,T), \\
u(0)=u_0, \quad u_t(0)=u_1, \quad u_{tt}(0)=u_2. &
\end{cases}
\end{equation}
We introduce the differential expressions $D_h=\partial_{t} + a\A$ and $D_w=\partial_{t}^2 + b\A \partial_t + c^2 \A$
corresponding to the heat equation and the Westervelt equation, respectively. Therewith, we represent 
\eqref{IBVP:abstract:interchange} as
\begin{equation*}
D_w w = -  f \qquad \text{ and } \qquad D_h u=w,
\end{equation*}
where later on we shall insert $f=(k (u_t)^2+s|\nabla u|^2)_{tt}$. For sufficiently smooth $u$ and $w= u_t + \A u$, we define the energies
\begin{align}
E_1[w](t)&=\tfrac{1}{2}\big( \hnorm{\A^{1/2}w_{tt}(t)}^2+\hnorm{\A^{1/2}w_t(t)}^2+\hnorm{\A w(t)}^2\big),\label{energy:west}\\
E_2[u](t)&= \tfrac{1}{2} \big( \hnorm{\A^{1/2}u_{ttt}(t)}^2+\hnorm{\A u_{tt}(t)}^2+\norm{\A^{3/2}u_{t}(t)}^2+ \hnorm{\A^{3/2}u}^2\big) \label{energy:u}
\end{align}
as well as the sum of \eqref{energy:west} and \eqref{energy:u},
\begin{equation}
\cE[u](t)= E_1[w](t) + E_2[u](t).\label{energy:sum}
\end{equation}
We mention that \eqref{energy:west} and \eqref{energy:u} are motivated by multiplication of $D_w w = -f$ and $D_h u = w$ with a 
suitable time derivative of $w$ and $u$, respectively, see \cite{KaLa09} or \eqref{id:heat1}.

\begin{lemma}[{\cite[Lemma 4.3]{BrKa14}}]
\label{lem:est1}
Provided $f \in H^1(0,T;\h)$, any solution $w$ of $D_w w=-f$ fulfills the estimate
\begin{equation}
\label{est1}
\begin{aligned}
&E_1[w](t)+ \check{b}_1 \int_0^t
\hnorm{w_{ttt}(\tau)}^2+\hnorm{\A^{1/2}w_{tt}(\tau)}^2+\hnorm{\A w_t(\tau)}^2+\hnorm{\A w(\tau)}^2\,d\tau\\
&\leq \check{c}_1 \left(E_1[w](0)+\int_0^t \hnorm{f_t(\tau)}^2+\hnorm{f(\tau)}^2\, d\tau\right)
\end{aligned}
\end{equation}
for $t\in (0,T)$ with $\check{b}_1>0$ sufficiently small and $\check{c}_1>0$ sufficiently large.
\end{lemma}

Now we use the following energy identity for the heat equation
\begin{equation}\label{id:heat1}
\begin{aligned}
&\int_0^t \hnorm{D_h v(\tau)}^2\, d\tau  = \int_0^t \hnorm{v_t(\tau)+a\A v(\tau)}^2\, d\tau \\
&=a\hnorm{\A^{1/2}v(t)}^2-a\hnorm{\A^{1/2}v(0)}^2+\int_0^t \hnorm{v_t(\tau)}^2+a^2\hnorm{\A v(\tau)}^2\,d\tau,
\end{aligned}
\end{equation}
where the second equality follows from $\hnorm{v_t(\tau)+a\A v(\tau)}^2 = \spr{v_t(\tau)+a\A v(\tau)}{v_t(\tau)+a\A v(\tau)}_\h$  and integration by parts. We apply \eqref{id:heat1} to $v=u_{ttt}$ (i.e. $D_h v=w_{ttt}$), $v=\A^{1/2} u_{tt}$ (i.e. $D_h v=\A^{1/2}w_{tt}$)
$v=\A u_t$ (i.e. $D_h v= \A w_t$) and $v= \A u$ (i.e. $D_h v = \A w$) and obtain
\begin{align*}
\int_0^t \hnorm{w_{ttt}(\tau)}^2\, d\tau =~&a\hnorm{\A^{1/2}u_{ttt}(t)}^2 -a\hnorm{\A^{1/2} u_{ttt}(0)}^2\\ 
&+\int_0^t \hnorm{u_{tttt}(\tau)}^2+a^2\hnorm{\A u_{ttt}(\tau)}^2\, d\tau,\\
\int_0^t \hnorm{\A^{1/2}w_{tt}(\tau)}^2\, d\tau =~&a\hnorm{\A u_{tt}(t)}^2-a\hnorm{\A u_{tt}(0)}^2\\ 
&+\int_0^t \hnorm{\A^{1/2} u_{ttt}(\tau)}^2+a^2\hnorm{\A^{3/2} u_{tt}(\tau)}^2\,d\tau,\\
\int_0^t \hnorm{\A w_t(\tau)}^2\, d\tau =~& a\hnorm{\A^{3/2}u_{t}(t)}^2-a\hnorm{\A^{3/2}u_{t}(0)}^2\\
&+\int_0^t \hnorm{\A u_{tt}(\tau)}^2+a^2\hnorm{\A^2 u_{t}(\tau)}^2\, d\tau,\\
\int_0^t \hnorm{\A w(\tau)}^2\, d\tau =~& a \hnorm{\A^{3/2} u(t)}^2 - a \hnorm{\A^{3/2} u(0)}^2\\
&+\int_0^t \hnorm{\A u_t(\tau)}^2 + a^2 \hnorm{\A^2 u(\tau)}^2\, d\tau.
\end{align*}
Therewith, by splitting the terms $\hnorm{w_{ttt}(\tau)}^2$, $\hnorm{\A w_{tt}(\tau)}^2$,  $\hnorm{\A^{1/2} w_t(\tau)}^2$ and $\hnorm{\A w(\tau)}^2$ on the left-hand side
of the estimate in Lemma \ref{lem:est1}, we arrive at 
\begin{align*}
& \cE[u](t) + \hat{b}_1 \int_0^t \big\{ \hnorm{w_{ttt}(\tau)}^2 + \hnorm{\A^{1/2} w_{tt}(\tau)}^2 + \hnorm{\A^{1/2} w_t(\tau)}^2+ \hnorm{\A w(\tau)}^2  \\
& \quad +\hnorm{u_{tttt}(\tau)}^2+ \hnorm{\A u_{ttt}(\tau)}^2 + \hnorm{\A^{1/2} u_{ttt}(\tau)}^2 + \hnorm{\A^{3/2} u_{tt}(\tau)}^2\\[3pt]
& \quad + \hnorm{\A u_{tt}(\tau)}^2 + \hnorm{\A^2 u_t(\tau)}^2 + \hnorm{\A u_t(\tau)}^2 + \hnorm{\A^2 u(\tau)}^2 \big\}\, d\tau\\
& \leq c_1 \left( \cE[u](0) + \int_0^t \left\{ \hnorm{f(\tau)}^2 + \hnorm{f_t (\tau)}^2 \right\} d\tau \right)
\end{align*}
where $b_1>0$ is sufficiently small and $c_1>0$ is sufficiently large. Defining
\begin{equation}
k[u](t) = \hnorm{u_{tttt}(t)}^2 + \hnorm{\A u_{ttt}(t)}^2 + \hnorm{\A^{3/2} u_{tt}(t)}^2 + \hnorm{\A^2 u_t(t)}^2 + \hnorm{\A^2 u(t)}^2
\end{equation}
and observing that, by Poincar\'e's inequality, we have $E_2[u] \leq (C_{H^1,L^2})^2 k[u]$ yields our next intermediate result.
\begin{lemma} Provided $f\in H^1(0,T;\h)$, any solution $u$ of $D_w D_h u = -f$ satisfies
\begin{equation}
\label{estu1}
\cE[u](t) +  \int_0^t  \cE[u](\tau) + k[u](\tau)\, d\tau \\
\leq \hat{c}_2 \left(\cE[u](0) + \int_0^t  \hnorm{f(\tau)}^2+ \hnorm{f_t(\tau)}^2 \, d\tau \right)
\end{equation}
for $t \in (0,T)$ and some $\hat{c}_2>0$ sufficiently large.
\end{lemma}

\begin{proposition}
\label{prop:linear}
Suppose that  $T$ is either finite or infinite. Suppose $f \in H^1(0,T;L_2(\Om))$ and $u_0 \in H^4(\Om)$, $u_1 \in H^3(\Om)$, $u_2 \in H^3(\Om)$ 
and $u_{ttt}(0) \in H^1(\Om)$, where 
\begin{equation*}
u_{ttt}(0)= f(0) + (a+b) \Delta u_{2} + c^2 \Delta u_{1} - ab \Delta^2 u_{1} - ac^2 \Delta^2 u_0
\end{equation*}
and let $u_0 |_\Gamma = \Delta u_0 |_\Gamma = u_1 |_\Gamma = \Delta u_1 |_\Gamma = u_2 |_\Gamma
= \Delta u_2 |_\Gamma = u_{ttt}(0) |_\Gamma =0$ in the sense of traces. 
Then \eqref{IBVP:Dir:hom:lin} has a unique solution 
\begin{align*}
u \in \mathcal{V}=~ & L_\infty(0,T;H^4(\Om)) \cap H^1(0,T; H^4(\Om)) \cap W_\infty^2(0,T;H^3(\Om)) \cap H^2(0,T;H^3(\Om)) \\
		&\cap W_\infty^3(0,T;H^1(\Om)) \cap  H^3(0,T;H^2(\Om)) \cap H^4(0,T;L_2(\Om)),
\end{align*} 
with $u |_\Gamma = \Delta u |_\Gamma = u_t |_\Gamma = \Delta u_t |_\Gamma = u_{tt} |_\Gamma
= \Delta u_{tt} |_\Gamma = u_{ttt} |_\Gamma =0$ in the trace sense.
\end{proposition}
\begin{proof}
Since the regularity of the initial values implies that $\cE[u](0)$ is finite and moreover, $f \in H^1(0,T;L_2(\Om))$, we conclude that
the right-hand side of \eqref{estu1} is bounded. Hence, from the fact that $\cE[u](t)$ is bounded for all $t \in (0,T)$, we obtain 
$$w \in L_\infty (0,T;H^2(\Om)) \qquad \text{and} \qquad w_t, w_{tt} \in L_\infty(0,T;H^1(\Om))$$ as well as 
$$u, u_t \in L_\infty(0,T;H^3(\Om)), \quad u_{tt} \in L_\infty(0,T;H^2(\Om)) \quad \text{and}\quad u_{ttt} \in L_\infty(0,T;H^1(\Om)).$$ This,
by $w= a\Delta u - u_t$, implies $$u \in L_\infty(0,T; H^4(\Om)) \cap W_\infty^2(0,T; H^3(\Om)) \cap W_\infty^3(0,T; H^1(\Om)).$$
The fact that $\int_0^t  \cE[u](\tau) d\tau$ is finite provides us with $$u \in L_2(0,T;H^4(\Om)) \cap H^2(0,T;H^3(\Om)) \cap H^3(0,T;H^1(\Om)).$$ From 
the boundedness of $\int_0^t k[u](\tau)\, d\tau$ we get $u, u_t \in L_2(0,T;H^4(\Om))$, $u_{tt} \in L_2(0,T;H^3(\Om))$, $u_{ttt} \in L_2(0,T;H^2(\Om))$
and $u_{tttt} \in L_2(0,T;L_2(\Om))$. Hence the boundedness of the integral on the left hand side gives 
$$u \in H^1(0,T; H^4(\Om)) \cap H^2(0,T;H^3(\Om)) \cap H^3(0,T;H^2(\Om)) \cap H^4(0,T;L_2(\Om)).$$

Finally, there exists at most one solution. This can be seen by considering two solutions of \eqref{IBVP:Dir:hom:lin}. Their difference $\hat{u}$
solves \eqref{IBVP:Dir:hom:lin} with $f=0$ and $u_0 = u_1 = u_2 =0$. According to \eqref{estu1} we then have
$E[\hat u](t) +  \int_0^t  E[\hat u](\tau) + k[\hat u](\tau)\, d\tau \leq 0$, hence $\hat{u}=0$ and we are done. 
\end{proof}

\section{Well-posedness and exponential decay} 
\label{sec:wp}
In the present section we will show well-posedness of \eqref{IBVP:Dir:hom}, that is, existence of a unique solution 
$u(t,x)$ of \eqref{IBVP:Dir:hom} depending continuously on the initial data $u_0$, $u_1$ and $u_2$, where $(t,x) \in (0,T) \times \Om$ for
some finite $T>0$ or for infinite $T$. Our plan is to employ Proposition \ref{prop:linear} in a fixed-point argument. 
To this end, we use the space 
\begin{align*}
\mathcal{V} &= L_\infty(0,T;H^4(\Om)) \cap H^1(0,T; H^4(\Om)) \cap W_\infty^2(0,T;H^3(\Om)) \cap H^2(0,T;H^3(\Om))\\
& \quad \cap W_\infty^3(0,T;H^1(\Om)) \cap H^3(0,T;H^2(\Om)) \cap H^4(0,T;L_2(\Om))
\end{align*} 
endowed with the norm 
\begin{align*}
\norm{\cdot}_\mathcal{V}^2 &= \norm{\cdot}_{L_\infty(0,T;H^4)}^2 +  \norm{\cdot}_{H^1(0,T;H^4)}^2 + \norm{\cdot}_{W_\infty^2(0,T;H^3)}^2 + \norm{\cdot}_{H^2(0,T;H^3)}^2 \\
& \quad + \norm{\cdot}_{W_\infty^3(0,T;H^1)}^2+ \norm{\cdot}_{H^3(0,T;H^2)}^2 + \norm{\cdot}_{H^4(0,T;L_2)}^2
\end{align*}
and apply the Banach fixed-point theorem to the map 
\begin{equation}
\begin{aligned}
\mathcal{T}: \mathcal{W} &\rightarrow \mathcal{V},\\
\varphi &\mapsto u
\end{aligned}
\end{equation}
where $\mathcal{W}\subset \mathcal{V}$ endowed with the norm $\norm{\cdot}_\mathcal{V}$ is given by 
\begin{equation*}
\mathcal{W}=  \{v \in \mathcal{V} \colon  \|v\|_{\tilde{\mathcal{V}}}^2 \leq \bar{a}\}
\end{equation*}
with
\begin{align*}
    \|v\|_{\tilde{\mathcal{V}}}^2 &= \max \{  \|v_{tttt}\|_{L_2(0,T;L_2)}^2,   \|v_{ttt}\|_{L_2(0,T;H^1)}^2,  \|v_{tt}\|_{L_2(0,T;H^1)}^2,  \|v_{tt}\|_{L_\infty(0,T;H^2)}^2, \\  
&  \qquad \qquad \|v_t\|_{L_2(0,T;H^1)}^2,  \|v_t\|_{L_\infty(0,T;H^3)}^2, \|v\|_{L_\infty(0,T;H^3)}^2 \}
\end{align*}
and $\bar{a}$ sufficiently small and where $u$ is a solution of
\begin{equation}
\label{IBVP:Dir:hom:phi}
\begin{cases}
\begin{aligned}
(a\Delta - \partial_t)(u_{tt}-b\Delta u_t - c^2 \Delta u)&= (k (\varphi_t)^2+s|\nabla \varphi|^2)_{tt}		&& \text{in } (0,T) \times \Omega,\\
(u,u_t,u_{tt})&=(u_0, u_1, u_2)														&& \text{on } \{t=0\}\times\Omega,\\
(u, \Delta u)&= (0,0)																&& \text{on } [0,T) \times \Gamma.
\end{aligned}
\end{cases}
\end{equation}

\begin{theorem}[Well-posedness] 
\label{thm:wp}
Suppose $u_0 \in H^4(\Om)$, $u_1 \in H^3(\Om)$ with $\|u_1\|_{L_\infty} < (2k)^{-1}$, $u_2 \in H^3(\Om)$ and $u_{ttt}(0) \in H^1(\Om)$, where 
\begin{equation*}
u_{ttt}(0) = (1-2k u_1)^{-1} [ (a+b) \Delta u_2 -ac^2\Delta^2 u_0 +c^2  \Delta \psi_1 - ab  \Delta^2 u_1 + 2k (u_2)^2 + 2 s|\nabla u_1|^2 + 2s\nabla u_2 \nabla u_0 ]
\end{equation*} 
such that $u_0 |_\Gamma = \Delta u_0 |_\Gamma = u_1 |_\Gamma = \Delta u_1 |_\Gamma = u_2 |_\Gamma = \Delta u_2 |_\Gamma = u_{ttt}(0) |_\Gamma =0$.
There exist $\kappa >0$ and $\bar{a}>0$ sufficiently small such that, if 
\begin{equation}
\label{thm:wpcond}
\norm{u_0}_{H^4} + \norm{u_1}_{H^3} +\norm{u_2}_{H^2} + \norm{u_{ttt}(0)}_{H^1} \leq \kappa,
\end{equation}
then for any $T>0$ there exists a unique solution 
\begin{equation}\label{abar}
u \in \mathcal{V} \quad  \text{ with }\quad  \|u \|_{\tilde{\mathcal{V}}}^2 \leq \bar{a}
\end{equation}
satisfying $u |_\Gamma = \Delta u |_\Gamma = u_t |_\Gamma = \Delta u_t |_\Gamma = u_{tt} |_\Gamma = \Delta u_{tt} |_\Gamma = u_{ttt} |_\Gamma =0$.
The solution depends continously on the initial data $u_0$, $u_1$, $u_2$ and $u_{ttt}(0)$ with respect to the topology induced by $\|\cdot \|_\mathcal{V}$.
\end{theorem}

\begin{remark} 
Note that $\|u_1\|_{L_\infty} < (2k)^{-1}$ can be achieved via the smallness condition \eqref{thm:wpcond} together with the embedding $H^3(\Om) \hookrightarrow L_\infty(\Om)$. Degeneracy is avoided, since $\|u_t\|_{L_\infty((0,T)\times \Om)}$ can be bounded away from $-(2k)^{-1}$ again via the embedding $H^3(\Om) \hookrightarrow L_\infty(\Om)$ and choosing $\bar{a}$ sufficiently small.

Moreover, note that $\|\cdot\|_{\tilde{\mathcal{V}}}$ is weaker than $\|\cdot \|_\mathcal{V}$ and the embedding constant for $\mathcal{V} \hookrightarrow \tilde{\mathcal{V}}$ is just $1$, in particular this constant is independent of $T$ and $\Om$.
\end{remark}

The proof mainly boils down to the verification of the assumptions of Banach's Fixed-Point Theorem. We show that (1) $\mathcal{T}$is a self-mapping on $\mathcal{W}$, that (2) $\mathcal{W}$ is closed, that (3) $\mathcal{T}\colon \mathcal{W} \rightarrow \mathcal{W}$ is a contraction and that (4) the solution depends continuously on the initial data. 
At this stage, let us point out that the smallness condition \eqref{thm:wpcond} is essential for the application of Banach's Fixed-Point Theorem.
\vspace{3mm}\newline
{\it Step 1: $\mathcal{T}$ is a self-mapping on $\mathcal{W}$.} The self-mapping property is achieved by only allowing for sufficiently small initial data.
\begin{lemma}\label{lem:estf} 
Suppose $\varphi \in \mathcal{W}$. Then $f=(k (\varphi_t)^2 +s |\nabla \varphi |^2)_{tt} \in H^1(0,T;L_2(\Om))$ 
and we have the estimate
\begin{equation} \label{estf}
\norm{f}_{L_2(0,T;L_2)}^2 + \norm{f_t}_{L_2(0,T;L_2)}^2 \leq c \|\varphi\|_{\tilde{\mathcal{V}}}^4 \leq c \bar{a}^2
\end{equation}
for some constant $c>0$.
\end{lemma}
\begin{proof}
Explicitly we have $f=2k(\varphi_{tt})^2 +2k \varphi_t \varphi_{ttt} + 2s|\nabla \varphi_t|^2+ 2s\nabla \varphi \nabla \varphi_{tt}$ and
$f_t= 6k\varphi_{tt} \varphi_{ttt} +2k \varphi_t \varphi_{tttt} +6s \nabla \varphi_t \nabla \varphi_{tt}+s \nabla \varphi \nabla \varphi_{ttt}$.
We prove that, provided $\varphi \in \mathcal{W}$, $f$ and $f_t$ are both in $L_2(0,T;L_2(\Om))$ and can be estimated in terms of 
$\bar{a}^2$.

First note that, since $\varphi \in \mathcal{W}$, we have $\varphi_{tt} \in L_\infty(0,T;H^2(\Om)) \cap L_2(0,T;H^3(\Om))$.
Using the embeddings $H^2(\Om) \hookrightarrow L_\infty(\Om)$ and $H^3(\Om) \hookrightarrow L_2(\Om)$ we further conclude that
$\varphi_{tt} \in L_\infty(0,T;L_\infty(\Om)) \cap L_2(0,T;L_2(\Om))$ and since moreover, $L_\infty(0,T;L_\infty(\Om))$ 
is an ideal in the space $L_2(0,T;L_2(\Om))$ we arrive at $(\varphi_{tt})^2 \in L_2(0,T;L_2(\Om))$. We estimate
\begin{align*}
\norm{(\varphi_{tt})^2}_{L_2(0,T;L_2)}^2
& \leq  \norm{\varphi_{tt}}_{L_\infty(0,T;L_\infty)}^2 \norm{\varphi_{tt}}_{L_2(0,T;L_2)}^2\\
& \leq (C_{H^2,L_\infty})^2 \,(C_{H^1,L_2})^2 \,  \norm{\varphi_{tt}}_{L_\infty(0,T;H^2)}^2 \norm{\varphi_{tt}}_{L_2(0,T;H^1)}^2\\
& \leq (C_{H^2,L_\infty})^2 \,(C_{H^1,L_2})^2 \, \bar{a}^2.
\end{align*}
Furthermore, $\varphi \in \mathcal{W}$ implies $\varphi_t \in L_\infty(0,T;H^3(\Om)) \hookrightarrow L_\infty(0,T;L_\infty(\Om))$ and moreover, we have that 
$\varphi_{ttt} \in L_2(0,T;H^2(\Om)) \hookrightarrow L_2(0,T;L_2(\Om))$. Therefore, $\varphi_t \varphi_{ttt} \in L_2(0,T;L_2(\Om))$ and 
\begin{align*}
\norm{\varphi_t \varphi_{ttt} }_{L_2(0,T;L_2)}^2&
\leq  \norm{\varphi_{t}}_{L_\infty(0,T;L_\infty)}^2 \norm{\varphi_{ttt}}_{L_2(0,T;L_2)}^2\\
&\leq (C_{H^3,L_\infty})^2 \, (C_{H^1,L_2})^2 \,  \norm{\varphi_{t}}_{L_\infty(0,T;H^3)}^2 \norm{\varphi_{ttt}}_{L_2(0,T;H^1)}^2\\
&\leq (C_{H^3,L_\infty})^2 \, (C_{H^1,L_2})^2 \, \bar{a}^2.
\end{align*}
For the third term of $f$ note that $\nabla \varphi_t \in L_\infty(0,T;H^2(\Om)) \cap L_2(0,T;H^3(\Om))$ which via 
$H^2(\Om) \hookrightarrow L_\infty(\Om)$ and $H^3(\Om) \hookrightarrow L_2(\Om)$  gives us
$|\nabla \varphi_t|^2 \in L_2(0,T;L_2(\Om))$ and 
\begin{align*}
\norm{\nabla \varphi_t \nabla \varphi_{t} }_{L_2(0,T;L_2)}^2 
&\leq  (C_{H^2,L_\infty})^2 \, \norm{\varphi_t}_{L_\infty(0,T;H^3)}^2 \norm{\varphi_t}_{L_2(0,T;H^1)}^2 \\
&\leq (C_{H^2,L_\infty})^2 \, \bar{a}^2.
\end{align*}
Using the embeddings $H^3(\Om) \hookrightarrow L_\infty(\Om)$ and $H^2(\Om) \hookrightarrow L_2(\Om)$ we see that 
$\nabla \varphi \nabla \varphi_{tt} \in L_2(0,T;L_2(\Om))$ and estimate
\begin{align*}
\norm{\nabla \varphi \nabla \varphi_{tt} }_{L_2(0,T;L_2)}^2 
& \leq (C_{H^2,L_\infty})^2 \, \norm{\varphi}_{L_\infty(0,T;H^3)}^2 \, \norm{\varphi_{tt}}_{L_2(0,T;H^1)}^2\\
& \leq (C_{H^2,L_\infty})^2 \, \bar{a}^2.
\end{align*}
Now we proceed with treating the terms contained in $f_t$. For the first one we have $\varphi_{tt} \varphi_{ttt} \in L_2(0,T;L_2(\Om))$,
where we used $\varphi_{tt} \in L_\infty(0,T;H^2(\Om)) \hookrightarrow L_\infty(0,T;L_\infty(\Om))$ and 
$\varphi_{ttt} \in L_2(0,T;H^2(\Om)) \hookrightarrow L_2(0,T;L_2(\Om))$. As a consequence,
\begin{align*}
\norm{\varphi_{tt} \varphi_{ttt} }_{L_2(0,T;L_2)}^2 
&\leq (C_{H^2,L_\infty})^2 \, (C_{H^1,L_2})^2 \,\norm{\varphi_{tt}}_{L_\infty(0,T;H^2)}^2 \norm{\varphi_{ttt}}_{L_2(0,T;H^1)}^2\\
&\leq (C_{H^2,L_\infty})^2 \,(C_{H^1,L_2})^2 \, \bar{a}^2.
\end{align*}
Moreover, from $\varphi \in \mathcal{W}$ we have $\varphi_t \in L_\infty(0,T;H^3(\Om)) \hookrightarrow L_\infty(0,T;L_\infty(\Om))$ 
as well as $\varphi_{tttt} \in L_2(0,T;L_2(\Om))$, therefore $\varphi_t \varphi_{tttt} \in L_2(0,T;L_2(\Om))$ and 
\begin{align*}
\norm{\varphi_{t} \varphi_{tttt} }_{L_2(0,T;L_2)}^2 
&\leq (C_{H^3,L_\infty})^2 \, \norm{\varphi_t}_{L_\infty(0,T;H^3)}^2 \norm{\varphi_{tttt}}_{L_2(0,T;L_2)}^2\\
&\leq (C_{H^3,L_\infty})^2 \, \bar{a}^2.
\end{align*}
Since we have $\nabla \varphi_t \in L_\infty(0,T;H^2(\Om)) \hookrightarrow L_\infty(0,T;L_\infty(\Om))$ and 
$\nabla \varphi_{tt} \in L_2(0,T;H^2(\Om)) \hookrightarrow L_2(0,T;L_2(\Om))$, the term $\nabla \varphi_t \nabla \varphi_{tt}$ 
is an element of $L_2(0,T;L_2(\Om))$ and can be estimated as
\begin{align*}
\norm{\nabla \varphi_t \nabla  \varphi_{tt} }_{L_2(0,T;L_2)}^2 
& \leq  \,(C_{H^2,L_\infty})^2 \,  \norm{\varphi_t}_{L_\infty(0,T;H^3)}^2 \,  \norm{\varphi_{tt}}_{L_2(0,T;H^1)}^2\\
& \leq  \,(C_{H^2,L_\infty})^2 \, \bar{a}^2.
\end{align*}
It is finally not surprising that also $\nabla \varphi \nabla \varphi_{ttt} \in L_2(0,T;L_2(\Om))$ and 
\begin{align*}
\norm{\nabla \varphi \nabla  \varphi_{ttt} }_{L_2(0,T;L_2)}^2 
& \leq (C_{H^2,L_\infty})^2 \, \norm{\varphi}_{L_\infty(0,T;H^3)}^2 \norm{\varphi_{ttt}}_{L_2(0,T;H^1)}^2\\
& \leq (C_{H^2,L_\infty})^2 \, \bar{a}^2,
\end{align*}
where we have have used that $\nabla \varphi \in L_\infty(0,T;H^3(\Om)) \hookrightarrow L_\infty(0,T;L_\infty(\Om))$ and 
$\nabla \varphi_{ttt} \in L_2(0,T;H^1(\Om)) \hookrightarrow L_2(0,T;L_2(\Om))$. 
Altogether, we have $f \in H^1(0,T;L_2(\Om))$ and estimate \eqref{estf} holds for some $c>0$ depending on 
the embedding constants appearing in the foregoing estimates. 
\end{proof}

We assume now $\varphi \in \mathcal{W}$ and use our energy estimate \eqref{estu1} as well as \eqref{estf} to obtain
\begin{equation*}
\cE[u](t) +  \int_0^t  \cE[u](\tau) + k[u](\tau)\, d\tau \leq \hat{c}_2 \left(\cE[u](0) + c \bar{a}^2 \right)
\end{equation*}
where we choose $\overline{a}\leq (2\hat{c}_2 c )^{-1}$. 
In the proof of Proposition \ref{prop:linear} we have already mentioned that 
\begin{equation}
\label{ce}
\cE[u](0) \leq c_e ( \norm{u_0}_{H^4}^2 + \norm{u_1}_{H^3}^2 + \norm{u_2}_{H^3}^2 + \norm{u_{ttt}(0)}_{H^1}^2)
\end{equation} 
for some constant $c_e>0$ depending on the constant of Poincar{\'e}'s inequality. Smallness of the initial data,
\begin{equation*}
\|u_0\|_{H^4}+\|u_1\|_{H^3}+\|u_2\|_{H^3}+ \| u_{ttt}(0) \|_{H^1} \leq \kappa,
\end{equation*}
ensures $\cE[u](0) \leq c_e \kappa$. Choosing $\kappa =\frac{\overline{a}}{2 \hat{c}_2 c_e}$, we finally get 
\begin{equation*}
\cE[u](t) +  \int_0^t  \cE[u](\tau) + k[u](\tau)\, d\tau \leq \bar{a}
\end{equation*}
which gives us at first $u \in \mathcal{V}$ since the left-hand side is finite (see also the proof of Proposition \ref{prop:linear}) and
moreover, $u \in \mathcal{W}$ by the definitions of $\cE[u]$ and $k[u]$. Hence $\mathcal{T}\mathcal{W}\subseteq\mathcal{W}$ 
under the above assumptions.
\vspace{5pt}\newline
{\it Step 2: $\mathcal{W}$ is a closed subset of $\mathcal{V}$.} Note that $\mathcal{W}$ is a closed ball in $\mathcal{V}$ with respect to $\|\cdot \|_{\tilde{\mathcal{V}}}$.
\vspace{5pt}\newline
{\it Step 3: $\mathcal{T}: \mathcal{W} \rightarrow \mathcal{W}$ is a contraction.}
In order to show contractivity of $\mathcal{T}: \mathcal{W} \rightarrow \mathcal{W}$, suppose $u^{(1)}$ and $u^{(2)}$ are two 
solutions of \eqref{IBVP:Dir:hom:phi} and  $u^{(1)} = \mathcal{T} \varphi^{(1)}$, $u^{(2)}=\mathcal{T} \varphi^{(2)}$. 
Then $\hat{u}=u^{(1)} -u^{(2)}$ and $\hat{\varphi} = \varphi^{(1)} - \varphi^{(2)}$ solve the equation
\begin{equation*}
\begin{cases}
\begin{aligned}
(a\Delta - \partial_t)(\hat{u}_{tt}-b\Delta \hat{u}_t - c^2 \Delta \hat{u})&= \hat{f}		&& \text{in } (0,T) \times \Omega,\\
(\hat{u},\hat{u}_t,\hat{u}_{tt})&=(0, 0, 0)												&& \text{on } \{t=0\}\times\Omega,\\
(\hat{u}, \Delta \hat{u})&= (0,0)														&& \text{on } [0,T) \times \Gamma,
\end{aligned}
\end{cases}
\end{equation*}
where 
\begin{equation}
\label{hatf}
\begin{aligned}
\hat{f}&=(k(\varphi_t^{(1)})^2 - k(\varphi_t^{(2)})^2 +s |\nabla \varphi^{(1)}|^2 - s|\nabla \varphi^{(2)}|^2)_{tt}\\
  &= (k^2(\varphi_t^{(1)} + \varphi_t^{(2)}) \hat{\varphi}_t + s^2(\nabla \varphi^{(1)} + \nabla \varphi^{(2)}) \nabla \hat{\varphi})_{tt}.
  \end{aligned}
\end{equation}
\begin{lemma} \label{lem:hatf}
Suppose $\varphi^{(1)}, \varphi^{(2)} \in \mathcal{W}$. Then $\hat{f}$ given by \eqref{hatf} is $H^1(0,T;L_2(\Om))$ and we have 
\begin{equation}\label{est:hatf}
\norm{\hat{f}}_{L^2(0,T;L_2(\Om))}^2 + \norm{\hat{f_t}}_{L^2(0,T;L_2(\Om))}^2
\leq 2 \overline{c} \bar{a} \norm{\hat{\varphi}}_{\mathcal{V}}^2 
\end{equation}
for some constant $\overline{c}>0$.
\end{lemma}
\begin{proof}
Explicitly, we have
\begin{align*}
\hat{f}&= 2k \hat{\varphi}_{tt}(\varphi_{tt}^{(1)}+\varphi_{tt}^{(2)}) + 2k \varphi_{ttt}^{(1)} \hat{\varphi}_t +  2k \varphi_{t}^{(2)} \hat{\varphi}_{ttt}\\
&\quad+2s \nabla \hat{\varphi}_t (\nabla \varphi_{t}^{(1)} + \nabla \varphi_{t}^{(2)}) +2s  \nabla \hat{\varphi}_{tt} \nabla \varphi^{(1)} +2s  \nabla \hat{\varphi} \nabla \varphi_{tt}^{(2)}
\end{align*}
and 
\begin{align*}
\hat{f}_t&= 2k ( \hat{\varphi}_{ttt} ( \varphi_{tt}^{(1)} + \varphi_{tt}^{(2)})+  \hat{\varphi}_{tt} ( \varphi_{ttt}^{(1)} + \varphi_{ttt}^{(2)})+  \varphi_{tttt}^{(1)} \hat{\varphi}_t
			+ \varphi_{ttt}^{(1)} \hat{\varphi}_{tt} +  \varphi_{tt}^{(2)}\hat{\varphi}_{ttt}+ \varphi_{t}^{(2)} \hat{\varphi}_{tttt})\\
	   & \quad + 2 s (\nabla \hat{\varphi}_{tt} \nabla(\varphi_t^{(1)}+\varphi_t^{(2)}) +  \nabla \hat{\varphi}_{t} \nabla(\varphi_{tt}^{(1)}+\varphi_{tt}^{(2)}) + 
	         \nabla \hat{\varphi}_{ttt} \nabla \varphi^{(1)})\\
						& \quad +2s(\nabla \hat{\varphi}_{tt} \nabla \varphi_t^{(1)} +   \nabla \hat{\varphi}_t \nabla \varphi_{tt}^{(2)} +
		 \nabla \hat{\varphi} \nabla \varphi_{ttt}^{(2)}).
\end{align*}
If $\varphi^{(1)}, \varphi^{(2)} \in \mathcal{W}$, then we have $\hat{\varphi} \in \mathcal{V}$. Moreover, 
$\hat{\varphi}_{tt} \in L_2(0,T;H^3(\Om)) \hookrightarrow L_2(0,T;L_2(\Om))$
and $\varphi_{tt}^{(i)} \in L_\infty(0,T;H^2(\Om)) \hookrightarrow L_\infty(0,T;L_\infty(\Om))$ for $i \in \{1,2\}$. 
Hence $\hat{\varphi}_{tt}(\varphi_{tt}^{(1)}+\varphi_{tt}^{(2)}) \in L_2(0,T;L_2(\Om))$ and 
\begin{align*}
&\norm{\hat{\varphi}_{tt}(\varphi_{tt}^{(1)}+\varphi_{tt}^{(2)})}_{L_2(0,T;L_2)}^2 \\
&\leq \norm{\hat{\varphi}_{tt}}_{L_2(0,T;L_2)}^2 (\norm{\varphi_{tt}^{(1)}}_{L_\infty(0,T;L_\infty)}+\norm{\varphi_{tt}^{(2)}}_{L_\infty(0,T;L_\infty)})^2\\
&\leq 2   \,(C_{H^1,L_2})^2 \, (C_{H^2,L_\infty})^2 \norm{\hat{\varphi}_{tt}}_{L_2(0,T;H^1)}^2 (\norm{\varphi_{tt}^{(1)}}_{L_\infty(0,T;H^2)}^2+\norm{\varphi_{tt}^{(2)}}_{L_\infty(0,T;H^2)}^2)\\
&\leq 4\,(C_{H^1,L_2})^2 \, (C_{H^2,L_\infty})^2 \bar{a} \norm{\hat{\varphi}}_{\mathcal{V}}^2.							
\end{align*}
On the example of the latter we shown how to treat the terms contained in $\hat{f}$ and $\hat{f}_t$. It can be shown analogously to the proof of Lemma 
\ref{lem:estf} that, provided $\varphi^{(1)}, \varphi^{(2)} \in \mathcal{W}$, each of the summands 
in $\hat{f}$ and $\hat{f}_t$ is contained in $L_2(0,T;L_2(\Om))$ and can (up to a constant) be estimated in terms of the right-hand side of \eqref{est:hatf}. 
\end{proof}
Now we observe that 
\begin{equation}
\label{est:contr}
 \norm{u}_\mathcal{V}^2 
\leq c_k \esssup_{t \in [0,T]} \left(\cE[u](t) + \int_0^t k[u](\tau) \, d\tau\right)
\end{equation}
for some $c_k>0$ depending on the constant $C_{H^1,L^2}$ from the Poincar{\'e} inequality. Since $\hat{u}(0)=0$, $\hat{u}_t(0)=0$ and $\hat{u}_{tt}(0)=0$ we therefore have 
by \eqref{estu1} and Lemma \ref{lem:hatf} that
\begin{align*}
\norm{\hat u}_\mathcal{V}^2 
&\leq c_k \esssup_{t \in [0,T]} \left(\cE[\hat{u}](t) + \int_0^t  \cE[\hat{u}](\tau) + k[\hat{u}](\tau)\, d\tau \right)\\
& \leq c_2 c_k \esssup_{t \in [0,T]} \left(\cE[\hat{u}](t) + \int_0^t \norm{f(\tau) }_{L_2} +  \norm{f_t(\tau) }_{L_2}\, d\tau \right)\\
& \leq 2  \hat{c}_2 c_k \overline{c} \bar{a} \norm{\hat{\varphi}}_{\mathcal{\tilde{V}}}^2 \leq 2   \hat{c}_2 c_k \overline{c} \bar{a} \norm{\hat{\varphi}}_{\mathcal{V}}^2.
\end{align*}
Choosing $\bar{a}< (2 \hat{c}_2 c_k \overline{c})^{-1}$ finally yields contractivity with respect to $\norm{\cdot}_\mathcal{V}$. 
\vspace{5mm}
\newline
\textit{Step 4: Continuous dependence of the solution on the initial data.} Let now $u$ be a solution of \eqref{IBVP:Dir:hom}
with initial data $u_0$, $u_1$, $u_2$ and let $\bar{u}$ be a solution of \eqref{IBVP:Dir:hom} with initial data $\bar{u}_0$, $\bar{u}_1$, $\bar{u}_2$.
Then $\hat{u} = u - \bar{u}$ is a solution of the initial boundary value problem 
\begin{equation*}
\begin{cases}
\begin{aligned}
(a\Delta - \partial_t)(\hat{u}_{tt}-b\Delta \hat{u}_t - c^2 \Delta \hat{u})&= f^{\hat{u}}		&& \text{in } (0,T) \times \Omega,\\
(\hat{u},\hat{u}_t,\hat{u}_{tt})&=(u_0 - \bar{u}_0, u_1 - \bar{u}_1, u_2 - \bar{u}_2)		&& \text{on } \{t=0\}\times\Omega,\\
(\hat{u}, \Delta \hat{u})&= (0,0)															&& \text{on } [0,T) \times \Gamma,
\end{aligned}
\end{cases}
\end{equation*}
where $f^{\hat{u}}=(k(u_t)^2 - k(\bar{u}_t)^2 +s |\nabla u|^2 - s|\nabla \bar{u}|^2)_{tt}$. With $\varphi^{(1)}$, $\varphi^{(2)}$ replaced by $u$, $\bar{u}$, Lemma \ref{lem:hatf}
implies that, provided $u, \overline{u} \in \mathcal{W}$, we have $f^{\hat{u}} \in H^1(0,T;L_2(\Om))$ and the estimate 
\begin{equation*}
\norm{f^{\hat{u}}}_{L_2(0,T;L_2)} + \norm{f_t^{\hat{u}}}_{L_2(0,T;L_2)} \leq 2 \bar{c} \bar{a}   \norm{\hat{u}}_{\tilde{\mathcal{V}}}^2 \leq 2  \bar{c} \bar{a}   \norm{\hat{u}}_{\mathcal{V}}^2
\end{equation*}
holds for some $\bar{c} >0$. This by \eqref{estu1} gives us
\begin{align*}
\norm{\hat u}_\mathcal{V}^2 &\leq c_k \esssup_{t\in[0,T]}\left(\cE[\hat{u}](t) + \int_0^t  \cE[\hat{u}](\tau) + k[\hat{u}](\tau)\, d\tau \right)\\
 &\leq  c_k \hat{c}_2 \left( \cE[\hat{u}](0) + \esssup_{t\in[0,T]} \int_0^t \norm{f^{\hat{u}}(\tau) }_{L_2} +  \norm{f_t^{\hat{u}}(\tau) }_{L_2}\, d\tau \right)\\
 & \leq c_k \hat{c}_2  \cE[\hat{u}](0) + 2 c_k \hat{c}_2   \overline{c}   \bar{a}  \norm{\hat{u}}_{\tilde{\mathcal{V}}}^2 \leq c_k \hat{c}_2  \cE[\hat{u}](0) + 2 c_k \hat{c}_2   \overline{c}   \bar{a}  \norm{\hat{u}}_{\mathcal{V}}^2.
 \end{align*}
Choosing $\bar{a} < (4 c_k \hat{c}_2   \overline{c} )^{-1}$ and recalling \eqref{ce} we arrive at
\begin{equation*}
\norm{\hat{u}}_{\mathcal{V}}^2 \leq 2 c_k \hat{c}_2 c_e (\norm{u_0-\bar{u}_0}_{H^4}^2+\norm{u_1-\bar{u}_1}_{H^4}^2+ \norm{u_2-\bar{u}_2}_{H^3}^2+  \norm{u_{ttt}(0)-\bar{u}_{ttt}(0)}_{H^1}^2).
\end{equation*}
Let now $u_0 \rightarrow \bar{u}_0$ in $H^4(\Om)$,  $u_1 \rightarrow \bar{u}_1$, $u_2 \rightarrow \bar{u}_2$ in $H^3(\Om)$ and $u_{ttt}(0) \rightarrow \bar{u}_{ttt}(0)$ in $H^1(\Om)$.
Then $\norm{u - \bar{u}}_{\mathcal{V}}^2 \rightarrow 0$. Hence the solution depends continuously on the data
with respect to the $\mathcal{V}$-topology. \vspace{3mm}\newline
Therewith the proof of Theorem \ref{thm:wp} is complete. Finally, we show exponential decay of solutions. During the proof we 
will employ a classical barrier argument which has already been used in other studies of models in nonlinear acoustics, see \cite{BrKa14}, \cite{BKR14}, \cite{KaLa09}, \cite{KaLa11}, \cite{KaLa12},
\cite{KLM12} or \cite{KLP12}.

\begin{theorem}[Exponential decay]
\label{thm:decay} Suppose $u_0 \in H^4(\Om)$, $u_1, u_2 \in H^3(\Om)$ with $\|u_1\|_{L_\infty} < (2k)^{-1}$, and $u_{ttt}(0) \in H^1(\Om)$, where
\begin{equation*}
u_{ttt}(0) = (1-2k u_1)^{-1} [ (a+b) \Delta u_2 -ac^2\Delta^2 u_0 +c^2  \Delta u_1 - ab  \Delta^2 u_1 + 2k (u_2)^2 + 2 s|\nabla u_1|^2 + 2s\nabla u_2 \nabla u_0 ]
\end{equation*}
such that $u_0 |_\Gamma = \Delta u_0|_\Gamma = u_1 |_\Gamma = \Delta u_1 |_\Gamma = u_2 |_\Gamma = \Delta u_2 |_\Gamma= u_{ttt}(0)|_\Gamma =0$.\\
If the initial data satisfy the smallness condition
\begin{equation}
\label{initial:global}
 \norm{u_0}_{H^4}^2 + \norm{u_1}_{H^4}^2 + \norm{u_2}_{H^3}^2 + \norm{u_{ttt}(0)}_{H^1}^2  \leq \rho
\end{equation}
for some sufficiently small $\rho >0$, then the global solution  decays exponentially fast to zero as time tends to infinity 
 in the sense that there exists some $\omega >0$ such that
\begin{equation}
\label{edecay}
\norm{u(t)}_{H^4}^2 + \norm{u_t(t)}_{H^3}^2 + \norm{u_{tt}(t)}_{H^3}^2 + \norm{u_{ttt}}_{H^1}^2 \leq C \E^{-\om t}, \qquad t > 0.
\end{equation}
\end{theorem}

\begin{lemma}\label{lem:energybck}
The estimate
\begin{equation}\label{energyestimate:final}
\cE[u](T) + \int_0^T  \cE[u](t)+k[u](t)\, dt
\leq \hat{c} \left(\cE[u](0)+\esssup_{s\in[0,T]} \cE[u](s)\int_0^T k[u](t)\, dt\right)
\end{equation}
holds with $\hat{c}>0$ sufficiently large.
\end{lemma}
\begin{proof}
Based on the energy estimate for the linearized equation \eqref{estu1}, 
the result follows from inspection of the proof of Lemma \ref{lem:estf} (choose $\varphi = u$ and use \eqref{est:contr}).
\end{proof}

\begin{proof}[Proof of Theorem \ref{thm:decay}] First we show that for all initial values satisfying 
\begin{equation}\label{ekappa}
\cE[u](0)\leq \eta
\end{equation}
with $\eta$ sufficiently small, 
\begin{equation}\label{kappa}
\eta \leq (4 \hat{c} \max\{1,\hat{c}\})^{-1},
\end{equation}
we get that for all $T>0$ 
\begin{equation}\label{est:e}
\cE[u](T)\leq 2\max\{1,\hat{c}\}\eta.
\end{equation}
We prove the claim by contradiction. To this end, we assume that there exists a finite time 
such that \eqref{est:e} is violated. We suppose that $T_0$ the minimal such time 
(observe that $T_0>0$ since $\cE[u](0)< 2\max\{1,\hat{c}\}\eta$) and have
\begin{equation}\label{ass:fail} 
\cE[u](T_0)\geq 2\max\{1,\hat{c}\}\eta.
\end{equation} 
Moreover, \eqref{est:e} holds for all $T\in(0,T_0)$. From Lemma \ref{lem:energybck} we get 
\begin{equation}\label{est:global}
\cE[u](T)+\int_0^T \cE[u](t)+k[u](t) \, dt \leq \hat{c} \left(\cE[u](0)+2\max\{1,\hat{c}\}\eta \int_0^T k[u](t)dt\right)
\end{equation}
for all $T\in(0,T_0)$ which, by $\cE[u](0)\leq\eta $ and $4 \hat{c} \max\{1,\hat{c}\}\eta \leq1$, gives
\begin{equation}\label{est:decay1}
\cE[u](T) + \tfrac{1}{2} \int_0^T \cE[u](t)+ k[u](t)  \, dt \leq \hat{c} \, \cE[u](0).
\end{equation}
Thus we have $\cE[u](T)\leq \hat{c} \eta$ for all $T\in(0,T_0)$ and hence by continuity $\cE[u](T_0)\leq \hat{c}\eta$
which is a contradiction to \eqref{ass:fail}. 
This proves that the bound \eqref{est:e} holds for all $T>0$ provided \eqref{ekappa} holds with \eqref{kappa}.\newline
For the condition of the initial values \eqref{initial:global} note that we recall \eqref{ce}. Hence, \eqref{initial:global} with 
$\rho = \frac{\eta}{c_e}$ ensures \eqref{ekappa}. 
By application of Poincar{\'e}'s inequality to $k[u]$ and use of \eqref{ce} we get 
\begin{align*}
&\int_0^t \norm{u(\tau)}_{H^4}^2 +  \norm{u_t(\tau)}_{H^3}^2 + \norm{u_{tt}(\tau)}_{H^3}^2 + \norm{u_{ttt}(\tau)}_{H^1}^2 \, d\tau \\
& \quad \leq c_d (\norm{u(0)}_{H^4}^2 +  \norm{u_t(0)}_{H^3}^2 + \norm{u_{tt}(0)}_{H^3}^2 + \norm{u_{ttt}(0)}_{H^1}^2)
\end{align*}
for some constant $c_d$. This by a standard argument leads to \eqref{edecay} with  $\om ={c_d}^{-1}$.
\end{proof}

\begin{remark} Again, we mention that $\|u_1\|_{L_\infty} < {(2k)^{-1}}$ can be achieved by $\|u_1\|_{H^3}^2 \leq \rho $ for 
$\rho$ sufficiently small and the embedding $H^3(\Om) \hookrightarrow L_\infty(\Om)$.
\end{remark}

\section{Conclusions and outlook}
Based on suitable energy estimates for the linearized equation we proved existence of a unique solution of \eqref{IBVP:Dir:hom} 
depending continously on the sufficiently small initial data by means of a fixed-point argument. Moreover, this solution decays exponentially fast to zero 
(with respect to the corresponding norms) as time tends to infinity.

Considering equation \eqref{IBVP:Dir:hom} together with application relevant boundary conditions (e.g., inhomogeneous 
Neumann boundary conditions for modeling excitation or absorbing boundary conditions for modeling boundary dissipation) 
will be subject of further research. Moreover, there will be a follow-up paper (based on the concept of maximal $L_p$-regularity) which addresses the issue of optimal regularity for the Blackstock--Crighton equation in $L_p$-spaces.

\bigskip
\noindent
{\bf Acknowledgments.}
The author wishes to thank Barbara Kaltenbacher for many helpful discussions and fruitful comments on a first 
version of this manuscript and the referees for lots of suggestions which have been very helpful for the improvement this work. Furthermore, the support by the Austrian Science Fund (FWF): P24970
and the Karl Popper Kolleg ``Modeling-Simulation-Optimization'' funded by the Alpen-Adria-Universit\"at Klagenfurt and by the Carinthian Economic Promotion Fund (KWF) is acknowledged. 
\bigskip

\bibliography{BlaCriKuz_lit}
\bibliographystyle{plain}

\end{document}